\documentclass{article}

\usepackage[american]{babel}
\usepackage{amsmath,amsfonts,amssymb}
\usepackage{amsthm}

\usepackage[style=american]{csquotes}

\usepackage[shortlabels]{enumitem}

\theoremstyle{plain}
\newtheorem{thm}{Theorem}[section]

\theoremstyle{plain}

\newtheorem{lemma}[thm]{Lemma}

\newtheorem{conjecture}[thm]{Conjecture}

\theoremstyle{definition}
\newtheorem{rem}[thm]{Remark}
\newtheorem{defn}[thm]{Definition}

\theoremstyle{remark}

\newcommand{\zz}{\mathbf{z}}
\newcommand{\xx}{\mathbf{x}}
\newcommand{\yy}{\mathbf{y}}
\newcommand{\uR}{\underline{\mathbb{R}}}
\newcommand{\RR}{\mathbb{R}}
\newcommand{\mv}{\mathbf{m}}
\newcommand{\ee}{\mathbf{e}}
\newcommand{\NN}{\mathbb{N}}
\newcommand{\oS}{\overline{S}}
\newcommand{\mol}{\overline{m}}
\newcommand{\mul}{\underline{m}}
\newcommand{\ww}{\mathbf{w}}

\title{Intertwining of maxima of sum of translates functions with nonsingular kernels}
\author{B\'alint Farkas, B\'ela Nagy and Szilárd Gy. R\'ev\'esz}
\date{}
\begin{document}

\maketitle

\begin{center}
Dedicated to the memory of \\
Yu.~N.~Subbotin
and
S.~A.~Telyakovskii, \\
excellent mathematicians and fine people
\end{center}

\begin{abstract}
In previous papers we investigated
so-called sum of translates functions
$F(\xx,t):=J(t)+\sum_{j=1}^n \nu_j K(t-x_j)$,
where $J:[0,1]\to \uR:=\RR\cup\{-\infty\}$ is
a ``sufficiently nondegenerate''
and upper-bounded ``field function'',
and $K:[-1,1]\to \uR$ is a fixed ``kernel function'',
concave both on $(-1,0)$ and $(0,1)$,
and also satisfying the singularity condition
$K(0)=\lim_{t\to 0} K(t)=-\infty$.
For node systems $\xx:=(x_1,\ldots,x_n)$
with $x_0:=0\le x_1\le\dots\le x_n\le 1=:x_{n+1}$,
we analyzed the behavior of the local maxima vector
$\mv:=(m_0,m_1,\ldots,m_n)$,
where $m_j:=m_j(\xx):=\sup_{x_j\le t\le x_{j+1}} F(\xx,t)$.

Among other results we proved a strong intertwining property:
if the kernels are also decreasing on $(-1,0)$
and increasing on $(0,1)$,
and the field function is upper semicontinuous,
then for any two different node systems there are $i,j$ such
that $m_i(\xx)<m_i(\yy)$ and $m_j(\xx)>m_j(\yy)$.

Here we partially succeed to extend this even to nonsingular kernels.


\end{abstract}

\section{Introduction}\label{sec:Introduction}

In our papers \cite{mikrofentongen,Homeo,Vegtelen}
we analyzed interval maxima vectors of sum of translates functions.
The very notion of the sum of translates functions
originates
from an ingenious paper of Fenton \cite{Fenton},
who himself worked out results on them for use in
his work proving a conjecture of Barry.
About the origins and wide range of applications, of the approach,
ranging from the strong polarization problem to
Chebyshev constants and Bojanov theorems
we refer the reader to the papers \cite{Homeo,Vegtelen}
as well as to
\cite{Bojanov1979,Bojanov1983}.

A function $K:(-1,0)\cup (0,1)\to \RR$ will be called a \emph{kernel function}
if it is concave on $(-1,0)$ and on $(0,1)$, and if it satisfies
\begin{equation*}
\lim_{t\downarrow 0} K(t) =\lim_{t\uparrow 0} K(t).
\end{equation*}
By the concavity assumption these limits exist,
and a kernel function has one-sided limits also at $-1$ and $1$.
We set
\begin{equation*}
K(0):=\lim_{t\to 0}K(t),\
K(-1):=\lim_{t\downarrow -1} K(t)
\text{ and }
K(1):=\lim_{t\uparrow 1} K(t).
\end{equation*}
Note explicitly that we thus obtain an
extended continuous
function $K:[-1,1]\to \RR\cup\{-\infty\}=:\uR$,
and that we have $\sup K<\infty$.
Also note that a kernel function is almost everywhere differentiable.

We say that the kernel function $K$ is \emph{strictly concave}
if it is strictly concave on both of
the intervals $(-1,0)$ and $(0,1)$.

Further, we call it \emph{monotone}\footnote{These conditions---and
more, like $C^2$ smoothness and strictly negative second
derivatives---were assumed on the kernel functions
in the
paper of Fenton \cite{Fenton}.}
if
\begin{equation}
\label{cond:monotone}\tag{M}
K \text{ is nonincreasing on } (-1,0) \text{ and nondecreasing on } (0,1).
\end{equation}
By concavity, under the monotonicity condition
\eqref{cond:monotone} the endpoint values $K(-1)$, $K(1)$
are also finite.
If $K$ is strictly concave, then \eqref{cond:monotone}
implies \emph{strict monotoni\-ci\-ty}:
\begin{equation}
\label{cond:smonotone}\tag{SM}
K \text{ is strictly decreasing on } [-1,0)
\text{ and strictly increasing on } (0,1],
\end{equation}
where we have extended the assertion to the finite endpoint values, too.

A kernel function $K$ is called \emph{singular} if
\begin{equation}\label{cond:infty}
\tag{$\infty$}
K(0)=-\infty.
\end{equation}

Let $n\in \NN=\{1,2,\dots\}$ be fixed.
We will call a function $J:[0,1]\to\uR$
an \emph{external $n$-field function}\footnote{Again,
the terminology of kernels and fields came to our mind by analogy,
which in case of the logarithmic kernel $K(t):=\log|t|$ and
an external field $J(t)$ arising from a weight $w(t):=\exp(J(t))$
are indeed discussed in logarithmic potential theory.
However, in our analysis no further potential theoretic notions
and tools will be applied.
This is so in particular because our analysis is far more general,
allowing different and almost arbitrary kernels and fields;
yet the resemblance to the classical settings of logarithmic potential
theory should not be denied.},
or---if the value of $n$ is unambiguous from the
context---simply a \emph{field function},
if it is bounded above  on $[0,1]$,
and it assumes finite values at more than  $n$ different points,
where we count the points $0$ and $1$ with
weight\footnote{The weighted counting makes a difference only
for the case when $J^{-1}(\{-\infty\})$ contains the two endpoints;
with only $n-1$ further interior points in $(0,1)$
the weights in this configuration add up to $n$ only,
hence the node system is considered inadmissible.} $1/2$ only,
while the points in $(0,1)$ are accounted for with weight $1$.
Therefore, for a field function $J$ the set  $(0,1)\setminus J^{-1}(\{-\infty\})$
has at least $n$ elements, and if it has precisely $n$ elements,
then either $J(0)$ or $J(1)$ is finite, too.

\medskip

Further, we consider the \emph{open simplex}
\[
S:=S_n:=\{ \yy=
(y_1,\dots,y_n)\in (0,1)^n,\: 0< y_1<\cdots <y_n<1\},
\]
and its closure the \emph{closed simplex}
\[
\overline{S}:=\{
\yy=
(y_1,y_2,\ldots,y_n)\in [0,1]^n:\
0\leq y_1\leq \cdots \leq y_n\leq 1\}.
\]

\medskip

For any given $n\in \NN$,
kernel function $K$, constants $\nu_j>0$, $(j=1,\ldots,n)$,
and a given field function $J$
we will consider the \emph{pure sum of translates function}
\begin{equation*}
f(\yy,t):=\sum_{j=1}^n \nu_j K(t-y_j)\quad (\yy\in \overline{S},\: t\in [0,1]),
\end{equation*}
and also the \emph{(weighted) sum of translates function}
\begin{equation*}
F(\yy,t):=J(t)+\sum_{j=1}^n \nu_jK(t-y_j)\quad (\yy\in \overline{S},\: t\in [0,1]).
\end{equation*}

Note that the functions $J, K$ can take the value $-\infty$,
but not $+\infty$,
therefore
both 
sum of translates functions can be defined meaningfully.
Furthermore, $f:\oS \times [0,1] \to \uR$
is extended continuous, and $F(\yy,\cdot)$ is not constant $-\infty$,
hence $\sup_{t\in[0,1]}F(\yy,t)>-\infty$
holds.\footnote{These require some careful considerations and
the assumed degree of nonsingularity of $J$ is in fact \emph{the exact condition} to ensure $F \not\equiv -\infty$. For details see \cite{Homeo,Vegtelen}.}

We introduce the \emph{singularity set} of the field function $J$ as
\begin{equation*}
X:=X_J:=\{t\in [0,1]~:~J (t)=-\infty\},
\end{equation*}
and note that the so-called \emph{finiteness domain} of $J$,
$X^c:=[0,1]\setminus X$ has cardinality exceeding $n$
(in the above described, weighted sense),
in particular $X\neq [0,1]$.
Similarly, the singularity set of $F(\yy,\cdot)$ is
\[
\widehat{X}:=
\widehat{X}(\yy):=
\{t\in[0,1]~:~F(\yy,t)=-\infty\} \varsubsetneq [0,1].
\]
Accordingly, an interval $I\subseteq [0,1]$
with $I\subseteq \widehat{X}(\yy)$ will be called \emph{singular}.

Writing $y_0:=0$ and $y_{n+1}:=1$  we also set for
each $\yy\in \overline{S}$ and $j\in \{0,1,\dots, n\}$
\begin{align*}
I_j(\yy):=[y_j,y_{j+1}], \qquad
m_j(\yy):=\sup_{t\in I_j(\yy)} F(\yy,t),
\end{align*}
and
\begin{align*}
\mol(\yy):=\max_{j=0,\dots,n} m_j(\yy)=\sup_{t\in [0,1]}F(\yy,t),
\qquad \mul(\yy):=\min_{j=0,\dots,n} m_j(\yy).
\end{align*}
Of interest are the simplex minimax
and simplex maximin values
which are defined as follows
\begin{equation*}
M(S):= \inf_{\yy \in S} \mol(\yy),
\qquad
m(S):= \sup_{\yy \in S}\mul(\yy).
\end{equation*}

As has been said above,
for each $\yy\in \oS$
we have that $\mol(\yy)=\sup_{t \in [0,1]} F(\yy,t)\in \RR$
is finite.
Observe that an interval $I\subseteq [0,1]$ is contained in $\widehat{X}(\yy)$,
i.e., $I$ is singular,
if and only if $F(\yy,\cdot)|_I\equiv -\infty$.
In particular $m_j(\yy)=-\infty$ exactly when $I_j(\yy)\subseteq \widehat{X}(\yy)$.
A node system $\yy$ is called \emph{singular}
if there is $j\in \{0,1,\dots,n\}$ with $I_j(\yy)$ singular,
i.e., $m_j(\yy)=-\infty$;
and a node system $\yy\in \partial S= \oS\setminus S$ is called \emph{degenerate}.

An essential role is played by the \emph{regularity set}
(set of regular node systems)
\begin{align*}
Y:=Y_n& :=Y_n(X):=
\{\yy\in \oS: \yy \text{ is nonsingular}\}
 \\
&=\{\yy\in \oS: I_j(\yy)\not\subseteq \widehat{X}(\yy)
\text{ for } j=0,1,\dots,n\}
\\
&=\{\yy\in \oS:
m_j(\yy)\neq-\infty \text{ for } j=0,1,\dots,n\}.
\end{align*}
An important observation is that the regularity
set does not depend on the kernel function $K$,
but on the set where $K$ is $-\infty$ (which is subset of $\{-1,0,1\}$).
Similarly, it only depends on the singularity set
$X_J$
of $J$,
but not on the actual function $J$ itself.
If
$K$ is nonsingular and is finite valued at
$\pm 1$, too,
then  we necessarily have $\widehat{X}=X_J$,
and the notion of singularity of intervals, hence of node systems,
becomes totally independent of the kernel $K$ itself.
On the other hand if the kernel $K$ is singular,
then all degenerate node systems are outright singular, thus $Y\subset S$.
Note also that we have $S\subset Y$ if and only if $X$
(or equivalently $\widehat{X}$,
which differs from it only by a finite number of points,
if at all) has empty interior.
In particular, if $X$ has empty interior and $K$ is singular, then $Y=S$.

We also introduce the \emph{interval maxima vector function}
\[
\mv(\ww):=(m_0(\ww),m_1(\ww),\ldots,m_n(\ww)) \in \uR^{n+1} \quad (\ww \in \oS).
\]
From the above it follows that for $\ww\in \oS$
we have $\mv(\ww)\ne (-\infty,\ldots,-\infty)$.

\begin{defn}[Intertwining of maxima]
\label{def:intertwining}
 Let $J:[0,1]\to \uR$ be an $n$-field function,
$K$ be a kernel function and $\nu_i>0$, $i=1,\ldots,n$ be given constants.
We say that for this system \emph{intertwining of maxima} holds,
if for any two different
regular node systems $\xx,\yy\in Y$
both $m_i(\xx)>m_i(\yy)$ and $m_j(\xx)<m_j(\yy)$
occur for some indices $i,j \in \{0,1,\ldots,n\}$.
\end{defn}

In our earlier papers we mentioned this same property
under the terminology that \enquote{majorization does not occur},
or simply
\enquote{nonmajorization property}.
Majorization means for two node systems $\xx, \yy$
that $\mv(\xx)\ge \mv(\yy)$,
i.e.~$m_i(\xx)\ge m_i(\yy)$ for all $i=0,1,\ldots,n$.

For singular kernels, like $\log|t|$,
the intertwining
property was established under suitable assumptions---see the
discussion in Section \ref{sec:earlierresults} below.
In particular, exponentiating one of our results,
we obtained the following---to the best of our knowledge, new---observation.
The sequence of (local) absolute value maxima
on $[0,1]$ of the monic polynomials $P(\xx,t):=\prod_{j=1}^n (t-x_j)$
for different node systems can never majorize each other.

In our earlier results on intertwining an essential role was played
by a so-called \enquote{Homeomorphism Theorem},
proved for singular kernels
in \cite{Homeo}.
If the considered kernels are nonsingular,
then this fundamental tool is no longer available.
Also, other properties, like continuity of the $m_i(\xx)$,
were obtained by heavy use of the singularity assumption \eqref{cond:infty}.
Therefore, it was not clear what happens for general,
not necessarily singular kernels.
On the other hand Fenton formulated his results in \cite{Fenton}
for possibly nonsingular kernels
(even if under several other restrictive assumptions),
so  investigating the nonsingular case came as natural.
Here is what we could prove.

\begin{thm}[Nonsingular intertwining]
\label{thm:Intertwining-123}
Let $n\in \{1,2,3\}$, $\nu_1,\dots,\nu_n>0$, let $K$ be a strictly concave and (strictly)
monotone \eqref{cond:smonotone} kernel function and
let $J$ be an upper semicontinuous $n$-field function.

Then intertwining of maxima holds,
and there is a unique  equioscillation point i.e.,
a node system $\ww$ for which $m_0(\ww)=m_1(\ww)=\dots=m_n(\ww)$.
\end{thm}

In other words, the first assertion of this theorem states that
majorization---$m_i(\xx) \ge m_i(\yy)$ for all $i=0,1,\ldots,n$---cannot hold,
unless $\xx=\yy$.
The existence of an equioscillation point follows
from our earlier results, recalled as Theorem \ref{thm:mikro} below.
\section{Some earlier results and a conjecture}
\label{sec:earlierresults}

Let us recall two of our earlier results on the behavior of $\mv$
in case of singular kernels
which put  into context
Theorem \ref{thm:Intertwining-123} and Conjecture \ref{conj:intertwining} below.
The first one is a combination of
Corollary 3.2 and Corollary 4.2 of \cite{Vegtelen}.
\begin{thm}
Let $n\in\NN$, $\nu_1,\ldots,\nu_n>0$, let $K$ be a singular \eqref{cond:infty}
and monotone \eqref{cond:monotone} kernel function, and
let $J$ be an upper semicontinuous $n$-field function.

Then $M(S)=m(S)$ and there exists some node system $\ww\in Y$ at which
 the simplex maximin and the simplex minimax
are
attained:
\[
\mul(\ww)=m(S)=M(S)=\mol(\ww).
\]
 The node system $\ww$ is an equioscillation point.

\medskip\noindent Moreover, there are no $\xx, \yy \in Y$
with $m_j(\xx)>m_j(\yy)$ for every $j\in \{0,1,\ldots,n\}$. For any equioscillation point $\ee$ we have $\mol(\ee)=M(S)$.
\end{thm}

\begin{rem}
\begin{enumerate}[(a)]
\item
If there is an equioscillation point, then $M(S)\leq m(S)$.
Indeed, if $\ee$ is an equioscillation point, then $\ee\in Y$
and $M(S)\le \mol(\ee)$, $\mol(\ee)=\mul(\ee)$, $\mul(\ee)\le m(S)$.
Note that, for this we do not use any of the special properties
of the kernel function or the field function.

\item
If there is
intertwining
on $\oS$, then
$M(S)\geq m(S)$.
Indeed, if $M(S)< m(S)$, that is $\inf_\xx \mol(\xx)< \sup_\yy \mul(\yy)$,
then there are $\xx,\yy\in \oS$ such that
$\mol(\xx)<\mul(\yy)$ which entails
that $\mv(\yy)$ strictly majorizes $\mv(\xx)$.
This observation again uses no properties
of $K$ and $J$.
\end{enumerate}
\end{rem}
The following, second theorem (Theorem 2.4 from \cite{mikrofentongen})
can be viewed as the sharpest result in this direction.

\begin{thm}\label{thm:mikro}
Let $n\in\NN$, $\nu_1,\ldots,\nu_n>0$, let
$K$ be a \emph{monotone} \eqref{cond:monotone} kernel function
and $J$ be an arbitrary $n$-field function.
Then $M(S)=m(S)$ and there exists some node system $\ww\in \oS$ at which the simplex minimax is attained:
\[
\mol(\ww)=M(S).
\]

\medskip\noindent Moreover, there are no $\xx, \yy \in Y$
with $m_j(\xx)>m_j(\yy)$ for every $j\in \{0,1,\ldots,n\}$.
For any equioscillation point $\ee$ we have $\mol(\ee)=M(S)$,
and if $J$ is upper semicontinuous or $K$ is singular,
then, in fact, there exists an equioscillation point.
\end{thm}

Based on these latter theorems and Theorem
\ref{thm:Intertwining-123},
we put forward the general case as a conjecture:

\begin{conjecture}[Nonsingular Intertwining]
\label{conj:intertwining}
Let $n\in \NN$, $\nu_1,\ldots,\nu_n>0$,
let $K$ be a strictly concave and (strictly) monotone \eqref{cond:smonotone}
kernel function
and let $J$ be an
upper semicontinuous $n$-field function.
The conclusion of Theorem \ref{thm:Intertwining-123} remains true
even if $n>3$, in particular for $\xx,\yy\in Y$ the
coordinatewise inequality $\mv(\xx)\le \mv(\yy)$ implies $\xx=\yy$.
\end{conjecture}

In \cite{Shi} Shi studied such intertwining type properties
and the relation to minimax problems, via a \enquote{homeomorphism property},
that is
established under strong differentiability conditions.
Such techniques are not applicable here
(we lack good differentiability properties of the functions $m_0,\dots,m_n$).
We refer to \cite{TLMS2018} for a comparison with Shi's result
(in the periodic case).

\begin{rem}
\begin{enumerate}[(a)]
	\item
	In Conjecture \ref{conj:intertwining} we need to restrict to $\xx,\yy\in Y$ as the following example shows.
Let $K(t):=\log|t|$, $J(t):=-\infty$, if $t<2/3$ and $J(t):=0$ if $t\ge 2/3$,
$n:=1$, $\xx=(1/3)$, $\yy=(2/3)$.
Note that $J$ is upper semicontinuous.
Then $f(\xx,t)>f(\yy,t)$ for $t\in(2/3,1]$.
We have $m_0(\xx)=m_0(\yy)=-\infty$ and $m_1(\yy)=\log|1/3|<m_1(\xx)=\log|2/3|$,
so $\mv(\xx)\ge \mv(\yy)$ and $\mv(\xx)\ne \mv(\yy)$, i.e.,  majorization occurs.
\item In general, monotonicity of $K$ is necessary to exclude majorization. Indeed, in Example 5.4 of \cite{Vegtelen} a non-monotone kernel $K$ is given such that with $J\equiv 0$ strict majorization occurs.
\end{enumerate}
\end{rem}

It should be also clarified whether in Theorem \ref{thm:Intertwining-123} the upper semicontinuity of $J$ is needed; our proof below uses this property. So we also ask the validity of the previous conjecture for not upper semicontinuous $J$.

\section{Some technical lemmas }

First, we recall a lemma from \cite{mikrofentongen}
(see Lemma 3.3 therein, but also Lemma 3.2 in \cite{TLMS2018}).

\begin{lemma}\label{lem:widening}
Let $K$ be any kernel function. Let $0\le \alpha<a<b<\beta \le 1$ and  $p, q >0$.
Set
\begin{equation}\label{eq:mudef}
\kappa:=\frac{p(a-\alpha)}{q(\beta-b)}.
\end{equation}
\begin{enumerate}[(a)]
\item
\label{parta}
If $K$ satisfies \eqref{cond:monotone} and $\kappa\geq 1$,
then for every $t\in [0,\alpha]$ we have
\begin{equation}\label{eq:wideninglemma}
pK(t-\alpha)+qK(t-\beta) \le pK(t-a)+qK(t-b).
\end{equation}
\item
\label{partb}
If $K$ satisfies \eqref{cond:monotone} and $\kappa\leq 1$,
then \eqref{eq:wideninglemma} holds for every $t\in [\beta,1]$.
\item
\label{partc}
If $\kappa=1$ (but
$K$ does not
necessarily satisfy \eqref{cond:monotone}),
then \eqref{eq:wideninglemma} holds for every $t\in [0,\alpha] \cup [\beta,1]$.
\item
\label{partd}
In case of a strictly concave kernel function \ref{parta}, \ref{partb}
and \ref{partc}
hold with strict inequality in \eqref{eq:wideninglemma}.
\item
\label{parte}
If $K$ is a monotone kernel function \eqref{cond:monotone} ,
then for every $t\in [a,b]$
\begin{equation*}
pK(t-\alpha)+qK(t-\beta) \geq  pK(t-a)+qK(t-b),
\end{equation*}
with strict inequality if $K$ is strictly monotone \eqref{cond:smonotone}.
\end{enumerate}
\end{lemma}

The following two lemmas
settle particular cases
of Theorem \ref{thm:Intertwining-123}, but for general $n$.
\begin{lemma}
\label{lemma:comparable}
Let $n\in \NN$ be arbitrary,
and assume that $K$ is a  strictly concave
and monotone kernel function.
Let $J$ be an upper semicontinuous $n$-field
function.

If $\xx, \yy \in Y$ and $\xx \le \yy$
in the sense that $x_i\le y_i$, $i=1,\ldots,n$,
and $\xx\ne\yy$,
then intertwining of maxima holds.
\end{lemma}
\begin{proof}
The monotonicity assumption \eqref{cond:monotone} provides for all $i=1,\ldots,n$
the inequalities $K(t-x_i)\le K(t-y_i)$ for $0\le t \le x_1$
and $K(t-x_i)\ge K(t-y_i)$ for $y_n \le t \le 1$;
moreover, the inequalities are strict whenever $x_i<y_i$ is strict,
for monotonicity has to be strict monotonicity in view of strict concavity.

Taking
the positive linear combination of
these inequalities and adding $J(t)$ to both sides,
we get that $F(\xx,t) \le F(\yy,t)$
for  $0\le t \le x_1$ and
$F(\xx,t) \ge F(\yy,t)$
for $y_n \le t \le 1$.
Moreover, as it is excluded that $x_i=y_i$ for all $i=1,\ldots,n$,
these inequalities have to be strict unless $J(t)=-\infty$.
Picking points\footnote{Here and throughout we use that
$J$, hence $F$ is upper semicontinuous and thus
on compact sets they have maximum points.
Upper semicontinuity is thus an indispensable assumption for our argument. }
$z\in [0,x_1]$ with $m_0(\xx)=F(\xx,z)$
and $w\in [y_n,1]$ with $m_n(\yy)=F(\yy,w)$,
the finiteness of $m_0(\xx), m_n(\yy)$
(following from the assumption $\xx, \yy \in Y$)
entails that $J(z), J(w) >-\infty$, hence we find
\[
m_0(\xx)=F(\xx,z)<F(\yy,z)\le
\max_{I_0(\xx)} F(\yy,\cdot) \le
\max_{I_0(\yy)} F(\yy,\cdot) =m_0(\yy)
\]
and similarly $m_n(\yy)=F(\yy,w) < F(\xx,w) \le m_n(\xx)$.
These
altogether show intertwining of maxima for $\xx$ and $\yy$.
\end{proof}

\begin{lemma}
\label{lemma:samenode}
Let $n\in \NN$ be arbitrary,
and assume that $K$ is a  strictly concave
and monotone kernel function.

Suppose that for any
upper semicontinuous $n-1$-field function $J^*$
and for $n-1$ nodes
we know that intertwining holds.

Let $J$ be an upper semicontinuous $n$-field function.

If $\xx\ne\yy \in Y$ consist of $n$ nodes each, and
there is $i$ such that $x_i=y_i$, then  intertwining holds
for $\xx$ and $\yy$.
\end{lemma}
\begin{proof}
We  apply the assertion
to the new data $n^*:=n-1$,
$J^*(x):=J(x)+\nu_iK(x-x_i)$---which is an
upper semicontinuous
 $n^*$-field function---and
$\nu_j^*:=\nu_j$
for $j<i$ and
$\nu_j^*:=\nu_{j+1}$ for
$i\le j<n$.
Put now $\xx^*:=(x_1^*,\ldots,x_{n-1}^*)$, where $x_j^*:=x_j$ for $j<i$
and $x_j^*:=x_{j+1}$ for
$i\le j <n$;
and  construct $\yy^*$ similarly.

For these systems it is easy to see that
\[
m_j^*(\xx^*)=
\begin{cases}
m_j(\xx), &\text{ if } j<i,
\\
\max(m_i(\xx),m_{i+1}(\xx)),
&\text{ if } j=i,
\\
m_{j+1}(\xx),  &\text{ if }
i<j<n,
\end{cases}
\]
and similarly for $\yy^*$.
So  $\xx^*\ne \yy^*$ are two different node systems,
and both are nonsingular.
Then the assumption
provides that we have both
$m_j^*(\xx^*)>m_j^*(\yy^*)$ and also
$m_k^*(\xx^*)<m_k^*(\yy^*)$ for some $j,k$.
The same inequality immediately follows between
the respective maxima for $\xx$ and $\yy$ unless $j=i$ or $k=i$.
(Note that both cannot happen.)
If e.g. $j=i$, then we only
see $\max(m_i(\xx),m_{i+1}(\xx)) > \max(m_i(\yy),m_{i+1}(\yy))$.
But then if the maximum $m_i^*(\xx^*)=m_i(\xx)$
then obviously we also have
$m_i(\xx)>m_i(\yy)$;
and the same way if the maximum is $m_i^*(\xx^*)=m_{i+1}(\xx)$
then $m_{i+1}(\xx)>m_{i+1}(\yy)$.
Similar argument works for the case $k=i$.

Therefore it follows that strong
intertwining of maxima holds for $\xx,\yy$, too.

\end{proof}

By the same argument, we can also obtain the following
which will not be used, however.
\begin{rem}
\label{remark:samenodenousc}
Let $n\in \NN$ be arbitrary,
and assume that $K$ is a  strictly concave
and monotone kernel function.

Suppose that for any
$n-1$-field function $J^*$
and for $n-1$ nodes
we know that intertwining holds.

Let $J$ be an $n$-field function.

If $\xx\ne\yy \in Y$ consist of $n$ nodes each, and
there is $i$ such that $x_i=y_i$, then  intertwining holds
for $\xx$ and $\yy$.
\end{rem}

\section{Proof of Theorem \ref{thm:Intertwining-123}}
By Theorem \ref{thm:mikro}, indeed,
there exists an equioscillation point,
since $J$ was
assumed
to be upper semicontinuous.

Observe  that the second assertion
of the theorem
is entailed by the first one:
if $\xx \ne \yy$ are equioscillating,
then either  $m_i(\xx)\le m_i(\yy)$ for all $i=0,1,\ldots,n$
or
$m_i(\xx)> m_i(\yy)$ for all $i=0,1,\ldots,n$,
whence majorization occurs, providing a contradiction.

Therefore the proof hinges upon excluding majorization
i.e.~proving intertwining of maxima.
In what follows, we will use without further reference
that an upper semicontinuous function attains its supremum on compact sets.

\subsection{The case of only one node}
If $n=1$ and $\xx:=(x), \yy:=(y)$ are two nodes (``node systems''),
then $\xx\ne \yy$ entails $x\ne y$, say $x<y$.
This case is immediately solved by Lemma \ref{lemma:comparable}, $n=1$.

\subsection{The $n=2$ case}
\label{sec:n2case}
If $n=2$, and $\xx, \yy \in Y$ are two node systems,
with say $x_1\le y_1$, then either also $x_2\le y_2$, and
Lemma \ref{lemma:comparable}
applies, or
we must have
$x_2>y_2$.
Moreover, if $x_1=y_1$, then we can refer back
to the above Lemma \ref{lemma:samenode}
and the already settled case $n=1$
to obtain intertwining of maxima for $\xx,\yy$.
So, without
loss of
generality, we
may assume that
\[
x_1<y_1<y_2<x_2 .
\]

Taking $\alpha:=x_1, a:=y_1, b:=y_2, \beta:=x_2$
and $p:=\nu_1, q:=\nu_2$ in the above Lemma \ref{lem:widening},
according to
\ref{parta}, \ref{partb} and
\ref{partd}
we are led to
\[
\nu_1K(t-x_1)+\nu_2K(t-x_2)
<
\nu_1K(t-y_1)+\nu_2K(t-y_2),
\]
either on $I_0(\xx)=[0,x_1]$ or on $I_2(\xx)=[x_2,1]$.

Consider the first case and pick some $z\in I_0(\xx)$ with $m_0(\xx)=F(\xx,z)$.
In view of $\xx\in Y$,
$m_0(\xx)>-\infty$,
so also the value of $J$ at $z$ satisfies
$J(z)>-\infty$.
Therefore
we have
\begin{align*}
m_0(\xx)=F(\xx,z)
&=
J(z)+ \nu_1K(z-x_1)+\nu_2K(z-x_2)
\\& <
J(z)+ \nu_1K(z-y_1)+\nu_2K(z-y_2)
=F(\yy,z)
\\& \le
\max_{I_0(\yy)} F(\yy,\cdot)=m_0(\yy),
\end{align*}
using that $z\in I_0(\xx)=[0,x_1]\subset [0,y_1]=I_0(\yy)$.

Similarly, in the other case
we find $m_2(\xx)<m_2(\yy)$.

On the other hand by strict monotonicity
and part \ref{parte}
of Lemma \ref{lem:widening},
we also have
\begin{equation}
\nu_1K(t-x_1)+\nu_2K(t-x_2) > \nu_1K(t-y_1)+\nu_2K(t-y_2)
\text{ for }
t\in [y_1,y_2].
\label{y1y2strictincrease}
\end{equation}
Let us take here a point $u\in I_1(\yy)$
with $m_1(\yy)=F(\yy,u)$; then by $\yy\in Y$
we also have that $F(\yy,u)$, hence also $J(u)$, are both finite.
So adding $J(u)$ to
\eqref{y1y2strictincrease} with $t=u$,
we get
$m_1(\yy)=F(\yy,u) < F(\xx,u) \le \max_{I_1(\xx)} F(\xx,\cdot) = m_1(\xx)$.

Therefore, we have completed proving
 intertwining of maxima for $\xx, \yy$.

\subsection{The $n=3$ case}
\label{sec:n3case}
Assume now $n=3$ and take two node systems $\xx, \yy \in Y$.
Again, if there are equal $i$th coordinates,
then induction settles the assertion
according to
Lemma \ref{lemma:samenode}
and Subsection \ref{sec:n2case}.
Further, if there is a coordinatewise ordering between the node systems,
then Lemma \ref{lemma:comparable} finishes the argument.

So we must have $x_i<y_i$ for some coordinates and
also $x_i>y_i$ for all the other ones,
each case occurring.
By interchanging the role of the two nodes,
we may settle with two $<$ and one $>$ inequality signs
among the respective coordinates.
Correspondingly, we will consider three cases according
to occurrence of the inequality $x_i>y_i$ at $i=1,2$ or $3$.

\bigskip\noindent
\textbf{Case 1} ($i=2$). $x_1<y_1<y_2<x_2<x_3<y_3$.

We will compare the values of the various $m_j(\xx), m_j(\yy)$
through use of the intermediate node systems
$\zz:=(y_1,x_2,x_3)$ and $\ww:=(y_1,y_2,x_3)$.

Observe that taking into account the condition \eqref{cond:monotone}
and $x_1<y_1$ we get $m_0(\zz)\ge m_0(\xx)$ and $m_j(\zz)\le m_j(\xx)$, $j=1,2,3$.

Further, in view of
Lemma \ref{lem:widening}
\ref{parta}, \ref{partb} and \ref{partd}
(with $[\alpha,\beta]=[y_2,y_3]\supset [a,b]=[x_2,x_3]$)
we have either on $[0,y_2]$ or on $[y_3,1]$ the inequality
\begin{equation*}
\nu_2K(t-y_2)+\nu_3K(t-y_3)
<
\nu_2K(t-x_2)+\nu_3 K(t-x_3).
\end{equation*}
Adding $J(t)+\nu_1K(t-y_1)$ we find either on $[0,y_2]$ or on $[y_3,1]$
the inequality
\[
F(\yy,t) \le F(\zz,t)
\text{ with strict inequality if }
J(t)>-\infty.
\]
Now we consider the
maxima
of the sum of translate function $F(\yy,\cdot)$
on various intervals:
these
maxima
(and then also the respective values of $J$) have
to be finite (for $\yy\in Y$ with finite $m_j$ values),
whence  at the points, where
maxima
are attained, one has even the strict inequality.
It follows that either $m_0(\yy)<m_0(\zz)$
and $m_1(\yy)<\max_{[y_1,y_2]} F(\zz,\cdot) \le m_1(\zz)$ or $m_3(\yy)<m_3(\zz)$.

In view of \eqref{cond:monotone} and
$x_1<y_1$, we have
$m_1(\yy)< m_1(\zz) \le m_1(\xx)$ in the first case and
$m_3(\yy)<m_3(\zz)\le m_3(\xx)$
in the second case,
so altogether $m_j(\yy)<m_j(\xx)$ holds either for $j=1$ or for $j=3$.

We show similar comparison for $\ww$ next.
By the monotonicity assumption  \eqref{cond:monotone} and $x_3<y_3$
we have $m_j(\ww)\le m_j(\yy)$, $j=0,1,2$ (and $m_3(\ww)\ge m_3(\yy)$).

According to Lemma \ref{lem:widening}
\ref{parta}, \ref{partb} and \ref{partd}
applied for $[\alpha,\beta]=[x_1,x_2]\supset [a,b]=[y_1,y_2]$,
we have either on $[0,x_1]$ or on $[x_2,1]$ the strict inequality
\begin{equation*}
\nu_1K(t-x_1)+\nu_2K(t-x_2)
<
\nu_1K(t-y_1)+\nu_2 K(t-y_2).
\end{equation*}
Adding $J(t)+\nu_3K(t-x_3)$
we find either on $[0,x_1]$ or on $[x_2,1]$ the inequality
\[
F(\xx,t) \le F(\ww,t)
\textrm{ with strict inequality if }
J(t)>-\infty.
\]
Considering the
maxima
of the sum of translates function $F(\xx,\cdot)$
on various intervals, and that the
maxima
$m_j(\xx)$ are finite (for $\xx \in Y$),
we conclude that at the points,
where these maxima are attained,
we have strict inequality.
It follows that either $m_0(\xx)<m_0(\ww)$,
or $m_2(\xx)<m_2(\ww)$ and $m_3(\xx)<m_3(\ww)$ simultaneously.

On combining the above we find that either $m_0(\xx)<m_0(\ww)\le m_0(\yy)$
or $m_2(\xx)<m_2(\ww)\le m_2(\yy)$.
That is, either $m_0(\xx)<m_0(\yy)$ or $m_2(\xx)<m_2(\yy)$ must hold.

Therefore, intertwining of maxima holds for $\xx, \yy \in Y$.

\bigskip\noindent
\textbf{Case 2} $(i=3$). $x_1<y_1, x_2 < y_2 < y_3 <x_3$.

One thing is immediate from (strict) monotonicity: we have $m_2(\xx)>m_2(\yy)$.

Therefore, the proof hinges upon showing a reverse inequality for some $j$.
We will distinguish two subcases in the proof of $m_j(\xx)< m_j(\yy)$ for some $j$.

\medskip\noindent
\textbf{Case 2.1.}
If $\nu_1(y_1-x_1)\ge \nu_3(x_3-y_3)$.

We apply Lemma \ref{lem:widening}
\ref{parta} and \ref{partd}
with $p:=\nu_1, q:=\nu_3, \alpha:=x_1, a:=y_1, b:=y_3, \beta:=x_3$.
In this case $\kappa$, as defined in \eqref{eq:mudef},
will satisfy $\kappa\ge 1$,
so
we obtain
\[
\nu_1K(t-x_1)+\nu_3K(t-x_3)
<
\nu_1 K(t-y_1) + \nu_3 K(t-y_3)
\qquad (t \in [0,x_1]).
\]
In view of monotonicity and
$x_2<y_2$, we
also have here
\[
\nu_2K(t-x_2)<\nu_2K(t-y_2) \qquad (t \in [0,x_1]).
\]
Adding these inequalities and then $J(t)$ to both sides leads to
\[
F(\xx,t)\le F(\yy,t) \
\text{ for } t \in [0,x_1]
\text{ with strict inequality if } J(t)>-\infty.
\]

So pick now a point $v\in [0,x_1]$ with $m_0(\xx)=F(\xx,v)$:
then with this $v$ $F(\xx,v)$,
hence also $J(v)$ is finite, and the above entails
\[
m_0(\xx)=F(\xx,v) < F(\yy,v) \le\max_{[0,x_1]} F(\yy,\cdot) \le m_0(\yy),
\]
providing us the required inequality with $j=0$ in this case.

\medskip\noindent
{\textbf{Case 2.2.}} If $\nu_1(y_1-x_1)< \nu_3(x_3-y_3)$.

In this subcase
let us define $\zz:=(x_1,y_2,z_3)$ with $z_3:=\frac{\nu_1}{\nu_3}(y_1-x_1)+y_3$.
Note that then $x_1<y_2<y_3<z_3<x_3$ holds,
so
the nodes of $\zz$ are listed in their natural order.

Applying Lemma \ref{lem:widening}
\ref{partc} and \ref{partd}
with $p:=\nu_1, q:=\nu_3, \alpha:=x_1, a:=y_1, b:=y_3, \beta:=z_3$,
we see $\kappa=1$ and so the
lemma provides inequalities \emph{on both sides}:
\[
\nu_1K(t-x_1)+\nu_3K(t-z_3) < \nu_1 K(t-y_1) + \nu_3 K(t-y_3)
\qquad (t \in [0,x_1]\cup[z_3,1]).
\]
Let us add
$\nu_2 K(t-y_2)+J(t)$
to both sides.
Since $y_2$
 is strictly between $x_1$ and $z_3$, we have
$K(t-y_2)>-\infty$
everywhere in $[0,x_1]\cup[z_3,1]$.
So, we  add
$\nu_2 K(t-y_2)+J(t)$,
a finite amount to the above inequality
whenever $J(t)>-\infty$.
This furnishes
for all $t \in [0,x_1]\cup[z_3,1]$
\[
F(\zz,t) \le F(\yy,t)
\text{ with strict inequality whenever }
J(t)>-\infty.
\]
Take now two points $u\in [0,x_1]$ and $v\in [z_3,1]$
where $m_0(\zz)=F(\zz,u)$ and $m_3(\zz)=F(\zz,v)$.
Note that the interval $I_0(\zz)=[0,x_1]=I_0(\xx)$ is nonsingular,
because $\xx \in Y$.
So is the interval
$I_3(\zz)=[z_3,1]\supset [x_3,1]=I_3(\xx)$,
because
$I_3(\xx)$
is nonsingular.
It follows that the
maxima
$m_0(\zz), m_3(\zz)$ are finite,
thus so are $J(u), J(v)$, too.
From these and $[0,x_1]\subset [0,y_1]$, $[z_3,1]\subset [y_3,1]$ we infer
\begin{align*}
m_0(\zz)&=F(\zz,u)<F(\yy,u)
\le \max_{[0,y_1]} F(\yy,\cdot) =m_0(\yy),
\\
m_3(\zz)&=F(\zz,v)<F(\yy,v)
\le \max_{[y_3,1]} F(\yy,\cdot) =m_3(\yy).
\end{align*}
Therefore, we have the strict inequalities $m_i(\zz) < m_i(\yy)$ for $i=0, 3$.

Next, we apply Lemma \ref{lem:widening}
\ref{parta} and \ref{partb}
with
$p:=\nu_2, q:=\nu_3, \alpha:=x_2, a:=y_2, b:=z_3, \beta:=x_3$.
We find that either on $[0,x_2]$ or on $[x_3,1]$ the inequality
\[
\nu_2K(t-x_2)+\nu_3K(t-x_3)
\le
\nu_2 K(t-y_2) + \nu_3 K(t-y_3)
\]
must hold.
Adding $J(t)+\nu_1K(t-x_1)$ thus leads to
\[
F(\xx,t) \le F(\zz,t)
\]
either for all $t\in [0,x_2]$ or for all $t\in [x_3,1]$.
Taking
maxima
on $[0,x_1]$ or on $[x_3,1]$
therefore furnishes either $m_0(\xx)\le m_0(\zz)$
or $m_3(\xx)\le \max_{[x_3,1]} F(\zz,\cdot) \le m_3(\zz)$.

Combining this with the above inequalities
$m_i(\zz) < m_i(\yy)$ ($i=0,3$),
we also obtain either $m_0(\xx)< m_0(\yy)$ or
$m_3(\xx) < m_3(\yy)$.
That is, we arrive at the desired inequality $m_j(\xx) < m_j(\yy)$
either for $j=0$ or for $j=3$.

\bigskip\noindent
\textbf{Case 3 ($i=1$).} $y_1<x_1 <x_2 < y_2, x_3 <y_3$.

Consider the modified system
$J^*(t):=J(1-t)$, $\nu_j^*:=\nu_{4-j}$
and $K^*(t):=K(-t)$ and the modified node systems
$\yy^*:=(y_1^*,y_2^*,y_3^*):=(1-x_3,1-x_2,1-x_1)$
and $\xx^*:=(x_1^*,x_2^*,x_3^*):=(1-y_3,1-y_2,1-y_1)$.

Let us then compute $F^*(\yy^*,s)$.
We find
\begin{align*}
F^*(\yy^*,s)&=J(1-s)+\nu_3K(-(s-y_1^*))+\nu_2K(-(s-y_2^*))+\nu_1K(-(s-y_3^*))
\\ & = J(1-s)+\nu_3K(1-s-x_3)+\nu_2K(1-s-x_2)+\nu_1K(1-s-x_1)
\\& =F(\xx,1-s).
\end{align*}
Similarly, $F^*(\xx^*,s)=F(\yy,1-s)$.
It follows that $m^*_j(\xx^*)=m_{4-j}(\yy)$ and
that $m^*_j(\yy^*)=m_{4-j}(\xx)$,
therefore the two function systems and nodes exhibit intertwining of maxima
precisely in the corresponding cases.

However, for the modified system we have
$y_1^*=1-x_3>1-y_3=x_1^*$, $y_2^*=1-x_2>1-y_2=x_2^*$
and $y_3^*=1-x_1<1-y_1=x_3^*$,
so  for these node systems we can return to Case 2.
Therefore, $\xx^*, \yy^*$ have strict intertwining of maxima,
hence so does $\xx,\yy$, too.

\qed


\providecommand{\bysame}{\leavevmode\hbox to3em{\hrulefill}\thinspace}
\providecommand{\MR}{\relax\ifhmode\unskip\space\fi MR }
\providecommand{\MRhref}[2]{%
  \href{http://www.ams.org/mathscinet-getitem?mr=#1}{#2}
}
\providecommand{\href}[2]{#2}

\medskip

\noindent
\hspace*{5mm}
\begin{minipage}{\textwidth}
\noindent
\hspace*{-5mm}Bálint Farkas\\
 School of Mathematics and Natural Sciences,\\
 University of Wuppertal\\
  Gau{\ss}stra{\ss}e 20\\
 42119 Wuppertal, Germany\\
\end{minipage}

\medskip

\noindent
\hspace*{5mm}
\begin{minipage}{\textwidth}
\noindent
\hspace*{-5mm}Béla Nagy\\
 Department of Analysis,\\
 Bolyai Institute, University of Szeged\\
 Aradi vértanuk tere 1\\
  6720 Szeged, Hungary\\
\end{minipage}

\medskip

\noindent
\hspace*{5mm}
\begin{minipage}{\textwidth}
\noindent
\hspace*{-5mm}
Szilárd Gy.{} Révész\\
 Alfréd Rényi Institute of Mathematics\\
 Reáltanoda utca 13-15\\
 1053 Budapest, Hungary \\
\end{minipage}

\end{document}